\documentclass[10pt]{article}

\usepackage{amsfonts}
\usepackage{amssymb}
\usepackage{amsmath}
\usepackage{amsthm}
\usepackage[all]{xy}

\usepackage[all]{xy}

\newtheorem{thm}{Theorem}[section]
\newtheorem{pro}[thm]{Proposition}
\newtheorem{cor}[thm]{Corollary}
\newtheorem{lem}[thm]{Lemma}

\theoremstyle{definition}

\newtheorem{rem}[thm]{Remark}       

\renewcommand{\square}{\relax}    

\newcommand{\p}{{\mathbb P}}

\newcommand{\an}{{\mathcal O}}

\begin{document}

\title{Varieties of Picard rank one as components of ample divisors}

\author{Andrea Luigi Tironi}

\date{}

\maketitle

\begin{abstract}
Let $\mathcal{V}$ be an integral normal complex projective variety
of dimension $n\geq 3$ and denote by $\mathcal{L}$ an ample line
bundle on $\mathcal{V}$. By imposing that the linear system
$|\mathcal{L}|$ contains an element $A=A_{1}+...+A_{r}, r\geq 1$,
where all the $A_{i}$'s are distinct effective Cartier divisors
with Pic$(A_i)=\mathbb{Z}$, we show that such a $\mathcal{V}$ is
as special as the components $A_i$ of $A\in |\mathcal{L}|$. After
making a list of some consequences about the positivity of the
components $A_i$, we characterize pairs $(\mathcal{V},
\mathcal{L})$ as above when either $A_1\cong\mathbb{P}^{n-1}$ and
Pic$(A_j)=\mathbb{Z}$ for $j=2,...,r,$ or $\mathcal{V}$ is smooth
and each $A_i$ is a variety of small degree with respect to
$[H_i]_{A_i}$, where $[H_i]_{A_i}$ is the restriction to $A_i$ of
a suitable line bundle $H_i$ on $\mathcal{V}$.
\end{abstract}

\section{Introduction}\label{intro}

Projective manifolds with an irreducible hyperplane section being a
special variety have been studied since longtime (see, e.g., \cite{BS} and \cite{F1}), but the corresponding study for a
reducible hyperplane section consisting of a simple normal crossing
divisor whose components are special varieties started only recently
by Chandler, Howard and Sommese \cite{CHS}. Therefore, we continue
here the study of varieties in terms of a hyperplane section $A$
which is not irreducible, assuming that $A$ is a union of distinct irreducible
components $A_1,...,A_r$, with $r\geq 1$. More precisely, let
$\mathcal{V}$ be an integral normal complex projective variety of
dimension $n\geq 3$ endowed with an ample line bundle $\mathcal{L}$.
Assume that

\begin{enumerate}
\item[$(\diamondsuit)$] $|\mathcal{L}|$ contains an element
$A=A_{1}+...+A_{r}, r\geq 1$, where all the components $A_{i}$ are
distinct and effective Cartier divisors with Pic$(A_i)$ of rank one.
\end{enumerate}

\noindent Let us observe here that this assumption is a natural generalization
of the classical hypothesis $\mathrm{Pic}(A')\cong\mathbb{Z}$ on a
hyperplane section $A'$ of $\mathcal{V}$. Furthermore, if $(\diamondsuit)$ holds
then every component $A_i$ of $A\in |\mathcal{L}|$ does not admit a non-trivial
morphism onto a variety $W_i$ with $0<\dim W_i<n-1$ and in general, also for the smooth case,
the main known results on reducible hyperplane sections (see, e.g., \cite{CHS}) can not be
applied on any of the $A_i$'s. However, we can show that if a reducible subvariety $A$
as in $(\diamondsuit)$ is contained in $\mathcal{V}$ as an ample divisor, then it
imposes severe restrictions to $\mathcal{V}$, that is, the
topological and geometric structures of $\mathcal{V}$ are very
closely related to those of each component of $A$.

So, in $\S$\ref{Proof} we prove first the following

\begin{thm}\label{thm-1}
Let $\mathcal{V}$ be an integral normal complex projective variety
of dimension $n\geq 3$ and let $\mathcal{L}$ be an ample line bundle
on $\mathcal{V}$. Assume that $(\diamondsuit)$ holds. Then all the
$A_i$'s are nef and big Cartier divisors on $\mathcal{V}$ and for
any $i=1,...,r$ there exist proper birational morphisms
$f_i:\mathcal{V}\to \mathcal{V}_i$ from $\mathcal{V}$ to a
projective normal variety $\mathcal{V}_i$ given by the map
associated to $|\mathcal{O}_{\mathcal{V}}(m_iA_i)|$ for some
$m_i>>0$. Furthermore, each map $f_i$ contracts at most a finite
number of curves on $\mathcal{V}$ and it is an isomorphism in a
neighborhood of $A_i$ such that $f_i(A_i)$ is an effective ample
divisor on $\mathcal{V}_i$.
\end{thm}

The above result shows that assumption $(\diamondsuit)$ implies
that all the components $A_i$ of $A$ are either ample, or at worst
big and $1$-ample in the sense of \cite[(1.3)]{S2}. By imposing
some restrictions on the singularities of $\mathcal{V}$, we are
able to deduce that all the $A_i$'s are in fact ample Cartier
divisors on $\mathcal{V}$.

\begin{thm}\label{thm-2}
Let $\mathcal{V}$ be an integral normal complex projective variety
of dimension $n\geq 3$ with at worst Cohen-Macaulay singularities.
Let $\mathcal{L}$ be an ample line bundle on $\mathcal{V}$ and
assume that $(\diamondsuit)$ holds. Furthermore, suppose that
$\mathcal{V}-F$ is a locally complete intersection for some
finite, possibly empty, set
$F\subset\mathcal{V}-\mathrm{Irr}(\mathcal{V})$ with
$\dim\mathrm{Irr}(\mathcal{V})\leq 0$, where
$\mathrm{Irr}(\mathcal{V})$ is the set of irrational singularities
of $\mathcal{V}$. Then all the $A_i$'s are ample Cartier divisors
on $\mathcal{V}$ and the maps $f_i:\mathcal{V}\to\mathcal{V}_i$ of
{\em Theorem \ref{thm-1}} are all isomorphisms.
\end{thm}

Furthermore, under some additional hypotheses on $\mathcal{V}$ and
on some $A_i$, we can finally obtain that
$\mathrm{Pic}(\mathcal{V})=\mathbb{Z}\langle\Lambda\rangle$ for an
ample line bundle $\Lambda$ on $\mathcal{V}$ (see Corollary
\ref{cor-1}).

All of these results allow us to list in $\S$\ref{consequences}
some consequences about the positivity of the $A_i$'s and to
obtain similar results as in \cite[Theorem 1]{Ba} (see also
\cite[Prop.VI]{S1}) for the case of reducible ample divisors on
$\mathcal{V}$ (see Propositions \ref{prop-2} and \ref{prop-3}).

We would like to note that the
above results make use of weak hypotheses on $\mathcal{V}$ and on
each $A_i$, and that $(\diamondsuit)$ seems optimal a priori for
Theorems \ref{thm-1} and \ref{thm-2}, since easy examples
show that these results do not hold assuming that
$\mathrm{Pic}(A_i)\neq\mathbb{Z}\langle\mathcal{H}_i\rangle$ for
some $i=1,...,r$, also when $r=2$ and $\mathcal{V}$ is smooth
(Remark \ref{remark-example}). Moreover, the techniques we employ in
$\S$\ref{Proof} leave out of account any special
polarization on each $A_i$ by a (very) ample line bundle on
$\mathcal{V}$ and they allow us to assume that $\mathcal{L}$ is
simply ample and not necessarily ample and spanned or very ample on
$\mathcal{V}$.

Finally, as a by-product of $\S$\ref{consequences} and some
results obtained by many other authors about smooth complex
projective variety $X$ containing ample divisors of special type
(e.g., \cite{Ba}, \cite{BS}, \cite{F1}, \cite{Ha1}, \cite{I2},
\cite{I1}, \cite{I}, \cite{S1}), in $\S$\ref{Fano}  we obtain
similar results as in \cite[Theorem 1]{Ba} and \cite[Prop.VI]{S1}
(see Propositions \ref{prop-2}, \ref{prop-3}), and in
$\S$\ref{small degree} we classify smooth polarized pairs $(X,L)$
which admit an ample divisor $A\in |L|$ such that $A=A_1+...+A_r,
r\geq 1,$ and all the components $A_i$ have small degree with
respect to suitable line bundles $H_i$ on $X$ for every
$i=1,...,r$.

\begin{pro}\label{prop-1}
Let $L$ be an ample line bundle on a smooth complex projective
variety $X$ of dimension $n$ with $n\geq 5$. Assume that there is a
divisor $A=A_1+...+A_r\in |L|, r\geq 1$, where each $A_i$ is an
irreducible and reduced normal Gorenstein projective variety with
$\dim\mathrm{Irr}(A_i)\leq 0$, where $\mathrm{Irr}(A_i)$ is the set
of irrational singularities of $A_i$. Suppose that for any $k=1,...,r$ there exist ample and spanned
line bundles $H_k$ on $X$ such that $[H_k]_{A_k}$ is very and
$[H_k]_{A_k}^{n-1}\leq 4$.
Then one of the following possibilities holds:
\begin{enumerate}
\item $r\geq 1$, $H_1=...=H_r=H$ and $(X,H)$ is one of the
following pairs:
\begin{enumerate}
\item $(\mathbb{P}^n,\mathcal{O}_{\mathbb{P}^n}(1))$ and $A_i\in
|\mathcal{O}_{\mathbb{P}^n}(a_i)|$ with $1\leq a_i\leq 4$ for every
$i=1,...,r$; \item $(\mathbb{Q}^n,\mathcal{O}_{\mathbb{Q}^n}(1))$
and $A_i\in |\mathcal{O}_{\mathbb{Q}^n}(a_i)|$ with $a_i=1,2$ for
every $i=1,...,r$; \item $X\subset\mathbb{P}^{n+1}$ is a
hypersurface of degree $3$ or $4$, and $A_i=H\in
|\mathcal{O}_{\mathbb{P}^{n+1}}(1)_X|$ for every $i=1,...,r$; \item
$X=\mathbb{Q}_1\cap\mathbb{Q}_2\subset\mathbb{P}^{n+2}$ is a
complete intersection of two quadric hypersurfaces
$\mathbb{Q}_i\subset\mathbb{P}^{n+2}$ for $i=1,2$, and $A_i=H\in
|\mathcal{O}_{\mathbb{P}^{n+2}}(1)_X|$ for every $i=1,...,r$; \item
$\pi:X\to\mathbb{P}^n$ is a double cover of $\mathbb{P}^n$ with
branch locus $\Delta\in |\mathcal{O}_{\mathbb{P}^n}(2b)|$,
$b=1,2$, $H\in |\pi^*\mathcal{O}_{\mathbb{P}^n}(1)|$ and
$A_i\in |\pi^*\mathcal{O}_{\mathbb{P}^n}(a_i)|, a_i=1,2$, for every
$i=1,...,r$; \item $\pi:X\to\mathbb{Q}^n$ is a double cover of a
quadric hypersurface $\mathbb{Q}^n\subset\mathbb{P}^{n+1}$ with
branch locus $\Delta\in |\mathcal{O}_{\mathbb{Q}^n}(2b')|$,
$b'=1,2$, and $A_i=H\in |\pi^*\mathcal{O}_{\mathbb{Q}^n}(1)|$
for every $i=1,...,r$; \item $\pi:X\to\mathbb{P}^n$ is a $d$-cover
of $\mathbb{P}^n$ with $d=3,4$, and $A_i=H\in
|\pi^*\mathcal{O}_{\mathbb{P}^n}(1)|$ for any $i=1,...,r$;
\end{enumerate}
\item $r\geq 2$, $X\cong\mathbb{P}^1\times\mathbb{P}^4$ and after
renaming
$(A_1,[H_1]_{A_1})\cong(\mathbb{P}^1\times\mathbb{P}^3,\mathcal{O}_{\mathbb{P}^1\times\mathbb{P}^3}(1,1))$,
with $A_1\in |\mathcal{O}_X(0,1)|$ and $H_1\in
|\mathcal{O}_X(1,1)|$, $(A_2,[H_2]_{A_2})\cong(\mathbb{P}^4,\mathcal{O}_{\mathbb{P}^4}(1))$
with $A_2\in |\mathcal{O}_X(1,0)|$ and $H_2\in |\mathcal{O}_X(t,1)|$ for some integer
$t\geq 1$, and the remaining polarized pairs $(A_k,[H_k]_{A_k})$ are of these two types.
\end{enumerate}
\end{pro}

\section{Notation}

Let $\mathcal{V}$ be an integral normal complex projective variety
of dimension $n\geq 3$ endowed with an ample line bundle
$\mathcal{L}$. All notation and terminology used here are standard
in algebraic geometry. We adopt the additive notation for line
bundles and the numerical equivalence is denoted by $\equiv$, while
the linear equivalence by $\simeq$~. The pull-back $\iota ^{\ast
}\mathcal{F}$ of a line bundle $\mathcal{F}$ on $\mathcal{V}$ by an
embedding $\iota :Y\hookrightarrow\mathcal{V}$ is sometimes denoted
by $\mathcal{F}_{Y}$. We denote by $K_{\mathcal{V}}$ the canonical
bundle of $\mathcal{V}$. For such a polarized variety
$(\mathcal{V},\mathcal{L})$, we will use the adjunction theoretic
terminology of \cite{BS} and we say that $\mathcal{V}$ is an
\textit{$n$-fold} (denoted by $X$) when $\mathcal{V}$ is smooth.

\section{Proof of Theorems \ref{thm-1} and \ref{thm-2}}\label{Proof}

\noindent First of all, let us show that $(\diamondsuit)$ implies
that all the components of $A$ are actually nef and big Cartier
divisors on $\mathcal{V}$, obtaining the following

\medskip

\noindent\textit{Proof of Theorem} \ref{thm-1}. Suppose that $A_1\cap A_2:=W$ is nonempty.
Define integers $a_{ij}$ putting
$\mathcal{O}_{\mathcal{V}}(A_i)_{A_j}:=\mathcal{O}_{A_j}(a_{ij}\mathcal{H}_j)$, where $\mathcal{H}_j$ is
the ample generator of Pic$(A_j)$ (mod torsion). Then
$$\mathcal{O}_W(a_{11}\mathcal{H}_1)=\mathcal{O}_W(A_1)=\mathcal{O}_W(a_{12}\mathcal{H}_2)$$
is an ample line bundle on $W$. Hence $a_{11}$ is positive. Since
$A$ is a connected divisor on $\mathcal{V}$ (see \cite[III
7.9]{HartBook}), for every $i=1,...,r$ there exists a component
$A_j$ of $A\in |\mathcal{L}|$ with $j\neq i$ such that $A_i\cap
A_j\neq \emptyset$. This shows that
$\mathcal{O}_{\mathcal{V}}(A_i)_{A_i}$ is ample for any $i=1,...,r$,
i.e. $A_i$ is nef and big. From \cite[III 4.2]{Ha1} (see also
\cite[(2.6.5)]{BS} or \cite[(1.2.30)]{LazBook}), it follows that
there exists a proper birational morphism $f_i:\mathcal{V}\to
\mathcal{V}_i$ from $\mathcal{V}$ to a projective normal variety
$\mathcal{V}_i$ given by the map associated to
$|\mathcal{O}_{\mathcal{V}}(m_iA_i)|$ for some $m_i>>0$, which is an
isomorphism in a neighborhood of $A_i$ and such that $f_i(A_i)$ is
an effective ample divisor on $\mathcal{V}_i$.

\medskip

\noindent\textit{Claim.} The morphism $f_i$ contracts at most a
finite number of curves on $\mathcal{V}$.

\smallskip

\noindent By \cite[(2.5.5)]{BS} note that $A_i\simeq H_i+D_i$ for
any $i=1,...,r$, where $H_i$ is $\mathbb{Q}$-ample and $D_i$ is
$\mathbb{Q}$-effective. This shows that $A_i\cdot A_j$ is nonzero
and so $A_i\cap A_j$ can not be empty. Let $W_{ij}:=A_i\cap A_j$ and
note that $W_{ij}$ is an ample divisor both in $A_i$ and in $A_j$.
Let $Z$ be an irreducible variety of dimension greater than or equal
to two. Since $\mathcal{L}$ is ample, we have that $Z\cap A_k\neq
\emptyset$ for some $k=1,...,r$, and then $W_{ik}\cap
Z\neq\emptyset$ since $\dim Z\cap A_k\geq 1$. Put $R:=Z\cap A_i$.
Thus $R$ is nonempty, $\dim R\geq 1$ and $\mathcal{O}(R)_R$ is
ample. By \cite[III 4.2]{Ha1} (or \cite[(1.2.30)]{LazBook}), the
linear system $|\mathcal{O}_{\mathcal{V}}(m_iA_i)|$ restricted to
$Z$ can not contract $Z$ to a lower dimensional variety.
\hfill{$\square$}

\bigskip

\noindent\textit{Proof of Theorem} \ref{thm-2}. Put
$\mathcal{O}_{\mathcal{V}}(A_i)_{A_j}=\mathcal{O}_{A_j}(a_{ij}\mathcal{H}_j)$
and $l_t:={\mathcal{O}_{A_t}(\mathcal{H}_t)}^{n-1}>0$ for any
$i,j,t\in\{1,...,r\}$. From Theorem \ref{thm-1} we know
that $a_{ii}>0$ for every $i=1,...,r$. Moreover, since $A$ is a
connected divisor on $X$ (see \cite[III 7.9]{HartBook}), for every
$j=1,...,r$ there exists a component $A_{k}$ of $A\in |L|$ with
$k\neq j$ such that $A_{j}\cap A_{k}\neq \varnothing .$ So we get
the following expressions
\begin{equation*}
A_{k}^{s}A_{j}^{n-s}=\mathcal{O}_{A_{k}}(a_{kk}\mathcal{H}_k)^{s-1}\mathcal{O}%
_{A_{k}}(a_{jk}\mathcal{H}_k)^{n-s}=a_{kk}^{s-1}a_{jk}^{n-s}l_{k}
\end{equation*}
\begin{equation}
A_{k}^{s}A_{j}^{n-s}=\mathcal{O}_{A_{j}}(a_{kj}\mathcal{H}_j)^{s}\mathcal{O}%
_{A_{j}}(a_{jj}\mathcal{H}_j)^{n-s-1}=a_{kj}^{s}a_{jj}^{n-s-1}l_{j},
\tag{1}
\end{equation}

\noindent with $1\leq s\leq n-1$. Note that $a_{jj}\neq 0$,
$a_{kk}\neq 0$, $a_{kj}>0$ and $a_{jk}>0$. Moreover, from the
equations (1) with $s=1,2$ we deduce that
\begin{equation*}
a_{jk}(a_{kj}^{2}a_{jj}^{n-3}l_{j})=a_{jk}(a_{kk}a_{jk}^{n-2}l_{k})=a_{kk}(a_{jk}^{n-1}l_{k})=a_{kk}(a_{kj}a_{jj}^{n-2}l_{j}),
\end{equation*}
that is,
\begin{equation}
a_{jk}a_{kj}=a_{jj}a_{kk}.  \tag{2}
\end{equation}

\noindent So we get

\begin{eqnarray*}
A_{s}^{2}\mathcal{L}^{n-2} &=&\sum_{h_{1}+...+h_{r}=n-2}\quad \frac{(n-2)!}{%
h_{1}!...h_{r}!}\quad A_{1}^{h_{1}}\text{ }...\text{ }A_{s}^{h_{s}+2}\text{ }%
...\text{ }A_{r}^{h_{r}} \\
&=&\sum_{h_{1}+...+h_{r}=n-2}\frac{(n-2)!}{h_{1}!...h_{r}!}\quad
a_{1s}^{h_{1}}\text{ }...\text{ }a_{ss}^{h_{s}+1}\text{ }...\text{ }%
a_{rs}^{h_{r}}l_{s} \\
&=&a_{ss}^{n-1}l_{s}+\text{[non-negative terms]}>0,
\end{eqnarray*}
for every $s=1,...,r$. Furthermore, for $j\neq k$ we have also the
following equations
\begin{equation}
A_{j}^{2}\mathcal{L}^{n-2}=[A_{j}]_{A_{j}}\mathcal{L}_{A_{j}}^{n-2}=\sum_{k_{1}+...+k_{r}=n-2}
\frac{(n-2)!}{k_{1}!...k_{r}!}
a_{1j}^{k_{1}}\text{ }...\text{ }a_{jj}^{k_{j}+1}\text{ }...\text{ }%
a_{rj}^{k_{r}}l_{j}  \tag{3}
\end{equation}
\begin{equation}
A_{k}^{2}\mathcal{L}^{n-2}=[A_{k}]_{A_{k}}\mathcal{L}_{A_{k}}^{n-2}=\sum_{h_{1}+...+h_{r}=n-2}
\frac{(n-2)!}{h_{1}!...h_{r}!}
a_{1k}^{h_{1}}\text{ }...\text{ }a_{kk}^{h_{k}+1}\text{ }...\text{ }%
a_{rk}^{h_{r}}l_{k}  \tag{4}
\end{equation}
\begin{equation}
A_{j}A_{k}\mathcal{L}^{n-2}=[A_{k}]_{A_{j}}\mathcal{L}_{A_{j}}^{n-2}=\sum_{k_{1}+...+k_{r}=n-2}
\frac{(n-2)!}{k_{1}!...k_{r}!}
a_{1j}^{k_{1}}\text{ }...\text{ }a_{jj}^{k_{j}}\text{ }...\text{ }%
a_{kj}^{k_{k}+1}\text{ }...\text{ }a_{rj}^{k_{r}}l_{j}  \tag{5}
\end{equation}
\begin{equation}
A_{j}A_{k}\mathcal{L}^{n-2}=[A_{j}]_{A_{k}}\mathcal{L}_{A_{k}}^{n-2}=\sum_{h_{1}+...+h_{r}=n-2}
\frac{(n-2)!}{h_{1}!...h_{r}!}
a_{1k}^{h_{1}}\text{ }...\text{ }a_{jk}^{h_{j}+1}\text{ }...\text{ }%
a_{kk}^{h_{k}}\text{ }...\text{ }a_{rk}^{h_{r}}l_{k}.  \tag{6}
\end{equation}

\bigskip

\noindent Since by (2) we get
\begin{equation*}
(a_{1j}^{k_1}\cdot ... \cdot a_{jj}^{k_j+1}\cdot ... \cdot
a_{rj}^{k_r})(a_{1k}^{h_1}\cdot ... \cdot a_{kk}^{h_k+1}\cdot ...
\cdot a_{rk}^{h_r})=
\end{equation*}
\begin{equation*}
=a_{jj}\ a_{kk}\ (a_{1j}^{k_1}\cdot ... \cdot a_{jj}^{k_j}\cdot
... \cdot a_{rj}^{k_r})(a_{1k}^{h_1} \cdot ... \cdot
a_{kk}^{h_k}\cdot ... \cdot a_{rk}^{h_r})=
\end{equation*}
\begin{equation*}
=a_{jk}\ a_{kj}\ (a_{1j}^{k_1}\cdot ... \cdot a_{jj}^{k_j}\cdot
... \cdot a_{rj}^{k_r})(a_{1k}^{h_1}\cdot ... \cdot
a_{kk}^{h_k}\cdot ... \cdot a_{rk}^{h_r})=
\end{equation*}
\begin{equation*}
=(a_{1j}^{k_1}\cdot ... \cdot a_{jj}^{k_j}\cdot ... \cdot
a_{kj}^{k_k+1}\cdot ... \cdot a_{rj}^{k_r})(a_{1k}^{h_1}\cdot ...
\cdot a_{jk}^{h_j+1}\cdot ... \cdot a_{kk}^{h_k}\cdot ... \cdot
a_{rk}^{h_r}),
\end{equation*}

\noindent from (3), (4), (5) and (6), we obtain that

\begin{equation*}
(A_{k}A_{j}\mathcal{L}^{n-2})^{2}=(A_{k}^{2}\mathcal{L}^{n-2})(A_{j}^{2}\mathcal{L}^{n-2})>0
\end{equation*}
for every $k$ and $j$ such that $A_{j}\cap A_{k}\neq \varnothing
$.

Thus, by \cite[(2.5.4)]{BS} we see that there exists a rational
number $\lambda _{jk}$ such that $A_{j}$ is numerically equivalent
to $\lambda _{jk}A_{k}$. Since from $(2)$ it follows that $A_k$
meets all the other components of $A$, by an inductive argument we
get
\begin{equation}
\mathcal{L}\simeq A_1+ ... +A_r \equiv (\lambda_{1k}+ ...
+\widehat{\lambda_{kk}} + ... +\lambda_{rk}+1)A_k=\mu_kA_k, \tag{7}
\end{equation}

\noindent where $\mu_k:=\lambda_{1k}+ ...
+\widehat{\lambda_{kk}} + ... +\lambda_{rk}+1$ and the symbol $\widehat{\ }$ denotes suppression.

Moreover, since
$$0<\mathcal{L}_{A_{j}}^{n-1}=A_{j}\mathcal{L}^{n-1}=\lambda
_{jk}A_{k}\mathcal{L}^{n-1}=\lambda
_{jk}\mathcal{L}_{A_{k}}^{n-1},$$ we have that $\lambda _{jk}>0$ for
every $j\neq k$ and from (7) it follows that $\mu_k\geq 1$, i.e.
$A_k$ is an ample Cartier divisor on $\mathcal{V}$ for any
$k=1,...,r$. By combining this with Theorem \ref{thm-1}, we obtain
that every $f_i$ is an isomorphism. \hfill{$\square$}

\begin{cor}\label{cor-1}
Let $\mathcal{V}$ be an integral normal complex projective variety
of dimension $n\geq 3$ with at worst Cohen-Macaulay singularities.
Let $\mathcal{L}$ be an ample line bundle on $\mathcal{V}$ and
assume that $(\diamondsuit)$ holds. Furthermore, suppose that one
of the following conditions holds:
\begin{enumerate}
\item[$(a)$] $\mathcal{V}$ is a locally complete intersection;
\item[$(b)$] $n=3$,
$\mathcal{V}-A_k$ and $\mathcal{V}-F$ are locally complete
intersections for some $k=1,...,r$ and some finite set
$F\subset\mathcal{V}-\mathrm{Irr}(\mathcal{V})$.
\end{enumerate}
Then $\mathrm{Pic}(\mathcal{V})=\mathbb{Z}\langle\Lambda\rangle$,
where $\Lambda$ is an ample line bundle on $\mathcal{V}$. In
particular, all the $A_i$'s are ample Cartier divisors on
$\mathcal{V}$ and the maps $f_i:\mathcal{V}\to\mathcal{V}_i$ of
Theorem \ref{thm-1} are all isomorphisms.
\end{cor}

\begin{proof}
In both cases $(a)$ and $(b)$, we simply use
Theorem \ref{thm-2} and \cite[(2.3.4)]{BS}.
\end{proof}

\begin{rem}\label{remark-example}
Let
$\mathcal{V}\cong\mathbb{P}(\mathcal{O}_{\mathbb{P}^{n-1}}\oplus\mathcal{O}_{\mathbb{P}^{n-1}}(-d))$
with $0<d<n$ and denote by $\pi :\mathcal{V}\to\mathbb{P}^{n-1}$ the
projection map. Note that there exists a smooth divisor $E$ on $\mathcal{V}$
which is a section of $\pi$ such that $E\cong\mathbb{P}^{n-1}$ and
$E_{E}\in |\mathcal{O}_{\mathbb{P}^{n-1}}(-d)|$. Put
$\mathcal{L}:=E+\pi^*\mathcal{O}_{\mathbb{P}^{n-1}}(a)$. Then, for a suitable
integer $a>>0$, we have that $\mathcal{L}$ is ample on $\mathcal{V}$ and that there
exists a divisor $A\in |\mathcal{L}|$ such that $A=A_1+A_2$, where $A_1=E$ and
$A_2\in |\pi^*\mathcal{O}_{\mathbb{P}^{n-1}}(a)|$ is a smooth
divisor on $\mathcal{V}$. This example shows that if $r\geq 2$ then the above
results can not be improved by assuming in $(\diamondsuit)$ that
$\mathrm{Pic}(A_i)\neq\mathbb{Z}$ for some $i=1,...,r$, also when
$r=2$ and $\mathcal{V}$ is a Fano $n$-fold with $n\geq 3$.
\end{rem}

\section{Some immediate consequences}\label{consequences}

Let us collect here some results due to the different positivity of
the components $A_i$ of $A\in |\mathcal{L}|$ under the hypothesis $(\diamondsuit)$.

\subsubsection{Nefness and bigness of the $A_i$'s}\label{subsection nef and big}

From Theorem \ref{thm-1}, we obtain the following

\begin{cor}\label{corollary nef and big}
Let $\mathcal{V}$ be an integral normal complex projective variety
of dimension $n\geq 3$ and let $\mathcal{L}$ be an ample line
bundle on $\mathcal{V}$. Assume that $(\diamondsuit)$ holds. Set
$D=\sum A_{i_h}$, where all the $i_h\in\{1,...,r\}$ are not
necessarily distinct indexes. Moreover, let
{\em{Irr}}$(\mathcal{V})$ be the set of irrational singularities
of $\mathcal{V}$. Then we have the following properties:
\begin{itemize}
    \item {\em{(Vanishing type Theorems).}} {\begin{enumerate} \item[(1)] Let $\varphi:\mathcal{V}\to Y$ be a morphism
    from $\mathcal{V}$ to a projective variety $Y$. Then $$\varphi_{(i)}(K_\mathcal{V}+D)=0 \quad\mathrm{ for }\ i\geq\max_{y\in\varphi({\mathrm{Irr}}(\mathcal{V}))}
    \dim (\varphi^{-1}(y)\cap\mathrm{Irr}(\mathcal{V}))+1,$$ where $\varphi_{(i)}(\mathcal{T})$ is the $i$-th higher derived functor of the direct
    image $\varphi_{*}(\mathcal{T})$ of a sheaf $\mathcal{T}$ on $\mathcal{V}$;
    \item[(2)] $H^1(\mathcal{V},-D)=0$; \item[(3)] assuming that
    \textrm{Irr}$(\mathcal{V})$ is finite and nonempty, we have that $$\dim H^0(\mathcal{V},
    K_\mathcal{V}+D)\geq \sharp (\mathrm{Irr}(\mathcal{V}))>0,$$ where $\sharp
    (\mathrm{Irr}(\mathcal{V}))$ is the number of points in {\em{Irr}}$(\mathcal{V})$;
    in particular, if $\mathcal{A}\in |D|$ lies in the set of Cohen-Macaulay
    points of $\mathcal{V}$, then $$\dim H^0(\mathcal{V},K_\mathcal{V})+\dim H^0(\mathcal{A},K_\mathcal{A})\geq\sharp
    (\mathrm{Irr}(\mathcal{V}))>0.$$
    \end{enumerate}}
    \item {\em{(Lefschetz type Theorems).}} {\begin{enumerate}
    \item[(4)] Let $\mathcal{A}\in |D|$ be a divisor such that $\mathcal{V}-\mathcal{A}$ is a local complete intersection. Then under the restriction map
    it follows that $$H^j(\mathcal{V},\mathbb{Z})\cong H^j(\mathcal{A},\mathbb{Z}) \qquad \textrm{ for }\ j\leq n-3,$$ and
    $$H^j(\mathcal{V},\mathbb{Z})\to H^j(\mathcal{A},\mathbb{Z}) \qquad \textrm{ for }\ j = n-2,$$ is injective with torsion free
    cokernel; moreover,
    we have that $\mathrm{Pic}(\mathcal{V})\cong\mathrm{Pic}(\mathcal{A})$ for $n\geq
    5,$ and the restriction mapping $\mathrm{Pic}(\mathcal{V})\to \mathrm{Pic}(\mathcal{A})$ is injective
    with torsion free cokernel for $n=4$; in particular, if $\mathcal{V}-A_i$
    is a local complete intersection for some $i=1,...,r$, then
    $\mathrm{Pic}(\mathcal{V})=\mathbb{Z}\langle\Lambda\rangle$ for $n\geq 4$ and some ample line bundle $\Lambda$ on $\mathcal{V}$;
    \end{enumerate}}
    \item {\em{(The Albanese mapping).}} {\begin{enumerate}
    \item[(5)] Assume that $\mathcal{A}\in |D|$ is normal. Moreover, suppose that $\mathcal{A}$ and $\mathcal{V}$ have at worst rational singularities. Then
    the map $\mathrm{Alb}(\mathcal{A})\to\mathrm{Alb}(\mathcal{V})$ induced by
    inclusion is an isomorphism. \end{enumerate}}
    \item {\em{(Hodge Index type Theorems).}} {\begin{enumerate}
    \item[(6)] For $n_{i_1}+ ... +n_{i_k}=n-1$ and
    $n_{i_1}\geq 1$, we have $$(A_{i_0}\cdot A_{i_1}^{n_{i_1}}\cdots A_{i_k}^{n_{i_k}})^2\geq
    (A_{i_0}^2\cdot A_{i_1}^{n_{i_1}-1}\cdots A_{i_k}^{n_{i_k}})(A_{i_1}^{n_{i_1}+1}\cdots
    A_{i_k}^{n_{i_k}}),$$ and for $n_{i_1}+ ... +n_{i_k}=n$, we get also $$(A_{i_1}^{n_{i_1}}\cdots A_{i_k}^{n_{i_k}})^n\geq
    (A_{i_1}^n)^{n_{i_1}}\cdots (A_{i_k}^n)^{n_{i_k}}>0,$$ where $i_h\in \{1,...,r\}$;
    \item[(7)] we have the following inequality: \ $(A_i^{n-1}\cdot A_j)(A_i\cdot A_j^{n-1})\geq A_i^n
    A_j^n>0$; \item[(8)] for $t\geq 1$ and any nef and big line
    bundle $H_i$ on $\mathcal{V}$ with $1\leq i\leq t$, it follows that
    \ $\mathcal{O}_\mathcal{V}(A_j)^{n-t}\cdot \prod_{i=1}^t H_i>0,$ where
    $j\in \{ 1,...,r \}$; \item[(9)] let $H$ be
    a line bundle such that $\dim H^0(\mathcal{V},NH)\geq 2$ for some $N\geq 1$. Then
    $H\cdot\mathcal{O}_\mathcal{V}(A_{i_1})\cdots\mathcal{O}_\mathcal{V}(A_{i_{n-1}})>0,$
    where $i_h\in \{1,...,r\}$.
    \end{enumerate}}
\end{itemize}
\end{cor}

\begin{proof}
First of all, from Theorem \ref{thm-1}
we deduce that the effective divisor $D=\sum A_{i_h}$
is nef and big on $\mathcal{V}$. Thus $\mathit{(1),
(2)}$ and $\mathit{(3)}$ of the statement follow from \cite[(2.2.5), (2.2.7), (2.2.8)]{BS}.
Finally, cases $\mathit{(4)}$ to $\mathit{(9)}$ follow from \cite[(2.3.3), (2.3.4), (2.4.4),
(2.5.1), (2.5.3), (2.5.2), (2.5.8) and (2.5.9)]{BS}.
\end{proof}

By applying Theorem \ref{thm-1} also to the zero locus
of special sections of $k$-ample vector bundles on an $n$-fold $X$
for $k\geq 0$ (see \cite[$\S 1$]{S2} and \cite[$\S 2.1$]{BS}), we
obtain the following

\begin{cor}\label{Corollary Pic-Theorem for vector bundles}
{\em (Lefschetz-Sommese type Theorem)} Let $\mathcal{E}$ be a
$k$-ample vector bundle of rank $r\geq 1$ on an $n$-fold $X$ with
$n\geq 4$. Assume that there exists a section $s\in\Gamma
(\mathcal{E})$ whose zero locus $Z=(s)_0$ is an integral normal
complex variety such that $\dim Z\geq 3$. Suppose that there exists a divisor $A=A_1+...+A_s,
s\geq 1$, on $Z$ which satisfies $(\diamondsuit)$. If $Z-A_i$ is a
local complete intersection for some $i=1,...,s$, then
$$H^i(X,\mathbb{Z})\cong
H^i(A_i,\mathbb{Z}) \qquad \textrm{ for }\ i\leq\min\{\dim
Z-3,n-r-k-1\},$$ and the restriction maps
$H^i(X,\mathbb{Z})\to H^i(A_i,\mathbb{Z})$ are injective with torsion free cokernel for
$i=\min\{\dim Z-2,n-r-k\}$.
\end{cor}

\begin{proof}
From Theorem \ref{thm-1}, we know that each $A_k$ is at
worst $1$-ample on $Z$. Thus by Corollary \ref{corollary nef and
big} (4), we have that
$$H^j(Z,\mathbb{Z})\cong H^j(A_i,\mathbb{Z}) \qquad \mathrm{ for
}\ \ j\leq \dim Z-3,$$ and the restriction maps
$H^j(Z,\mathbb{Z})\to H^j(A_i,\mathbb{Z})$ are injective with torsion free cokernel for
$j=\dim Z-2$. Moreover, from the Lefschetz-Sommese's Theorem for
$k$-ample vector bundles on an $n$-fold $X$
(see \cite[(1.16)]{S2} and \cite[(7.1.1), (7.1.9)]{LazBookbis}),
we deduce that
$$H^j(Z,\mathbb{Z})\cong H^j(X,\mathbb{Z}) \qquad \mathrm{ for
}\ \ j\leq n-r-k-1,$$ and that the restriction maps
$H^j(X,\mathbb{Z})\to H^j(Z,\mathbb{Z})$ are injective with torsion free cokernel for
$j=n-r-k$.
\end{proof}

\subsubsection{Ampleness of the $A_i$'s}\label{subsection ampleness}

Under the same assumption of Theorem
\ref{thm-2}, we first deduce the following

\begin{cor}\label{corollary ample}
Let $\mathcal{V}$ be an integral normal complex projective variety
of dimension $n\geq 3$ with at worst Cohen-Macaulay singularities.
Let $\mathcal{L}$ be an ample line bundle on $\mathcal{V}$ and
assume that $(\diamondsuit)$ holds. Moreover, suppose that $\mathcal{V}-F$ is a
local complete intersection for some finite, possibly empty, set
$F\subset \mathcal{V}-${\em Irr}$(\mathcal{V})$ with
$\dim\mathrm{Irr}(\mathcal{V})\leq 0$, where {\em
Irr}$(\mathcal{V})$ is the set of irrational singularities of
$\mathcal{V}$. Set $D=\sum A_{i_h}$, where all the
$i_h\in\{1,...,r\}$ are not necessarily distinct indexes. Then we
have the following properties:
\begin{itemize}
\item[$(I)$] {\em{(Fujita's Vanishing type Theorem).}} Given any
coherent sheaf $\mathcal{F}$ on $\mathcal{V}$, there exists an
integer $m(\mathcal{F},D)$ such that
$$H^i(\mathcal{V},\mathcal{F}\otimes\mathcal{O}_{\mathcal{V}}(mD+N))=0 \qquad \textrm{ for }\ i>0,
m\geq m(\mathcal{F},D),$$ where $N$ is any nef divisor
on $\mathcal{V}$; \item[$(II)$] $\dim H^0(\mathcal{V},D)\leq D^n+n$,
with equality if and only if $\mathcal{V}$ is one of the following:
$(a)$ $\mathbb{P}^n$; $(b)$ a quadric hypersurface
$\mathbb{Q}^n\subset\mathbb{P}^{n+1}$; $(c)$ a
$\mathbb{P}^{n-1}$-bundle over $\mathbb{P}^1$; $(d)$ a generalized cone over a
smooth submanifold $V\subset \mathcal{V}$ as in $(a), (b), (c)$;
\item[$(III)$] {\em{(Lefschetz type Theorems).}} Let $\mathcal{A}\in
|D|$ be a divisor such that $\mathcal{V}-\mathcal{A}$ is a local
complete intersection. Then under the restriction map it follows
that
$$H^j(\mathcal{V},\mathbb{Z})\cong H^j(\mathcal{A},\mathbb{Z}) \qquad \textrm{ for }\
j\leq n-2,$$ and
$$H^j(\mathcal{V},\mathbb{Z})\to H^j(\mathcal{A},\mathbb{Z}) \qquad \textrm{ for }\ j = n-1,$$ is injective with torsion free
cokernel; moreover, we have that
$\mathrm{Pic}(\mathcal{V})\cong\mathrm{Pic}(\mathcal{A})$\ for $n\geq
4$, and the restriction mapping \ $\mathrm{Pic}(\mathcal{V})\to
\mathrm{Pic}(\mathcal{A})$\ is injective with torsion free
cokernel for $n=3$; in particular, if $\mathcal{V}-A_i$ is a local
complete intersection for some $i=1,...,r$, then $\mathrm{Pic}(\mathcal{V})=\mathbb{Z}\langle\Lambda\rangle$
for $n\geq 3$ and some ample line bundle $\Lambda$ on $\mathcal{V}$.
\end{itemize}
\end{cor}

\begin{proof}
By Nakai-Moishezon-Kleiman criterion and
Theorem \ref{thm-2}, we see that $D=\sum A_{i_h}$ is
ample on $\mathcal{V}$. Thus case $\mathit{(I)}$ of the statement
follows from \cite[$\S 1$, Th.1]{F2} (see also
\cite[(1.4.35)]{LazBook}. Finally, we obtain case $\mathit{(II)}$ by
\cite[I (4.2), (5.10), (5.15)]{F1}, while $\mathit{(III)}$ follows
from \cite[(2.3.3)]{BS} and \cite[(2.3.4)]{BS} respectively.
\end{proof}

Finally, let us deduce also the following result for ample vector
bundles on a smooth variety.

\begin{cor}\label{cor rho=1 for vector bundles}
Let $\mathcal{E}$ be an ample vector bundle of rank $r\geq 1$ on an
$n$-fold $X$ with $n\geq 4$. Assume that there exists a section
$s\in\Gamma (\mathcal{E})$ whose zero locus $Z=(s)_0$ is a smooth
submanifold of $X$. Suppose that there exists a divisor $A=A_1+...+A_s, s\geq
1$, on $Z$ which satisfies $(\diamondsuit)$. If $r<n-2$ then both
$\mathrm{Pic}(X)$ and
$\mathrm{Pic}(Z)$ have rank one.
\end{cor}

\begin{proof}
It follows easily from Corollary
\ref{corollary ample} \textit{(III)} and \cite[(1.16)]{S2}, or
\cite[(7.1.5)(ii)]{LazBookbis}.
\end{proof}

\section{Some applications}\label{applications}

Here are two applications.

\subsection{All the $A_i$'s are Fano varieties of Picard rank one}\label{Fano}

First of all, let us prove the following

\begin{lem}\label{Lemma}
Let $\mathcal{L},\mathcal{V},A_i,\mathcal{V}_i,f_i$ be as in Theorem \ref{thm-1}. If $\mathcal{V}_k$ is
$\mathbb{Q}$-factorial for some $k=1,...,r$, then $A_k$ is ample on $\mathcal{V}$ and $f_k:\mathcal{V}\to\mathcal{V}_k$
is in fact an isomorphism.
\end{lem}

\begin{proof}
Take any line bundle $\mathcal{D}$ on
$\mathcal{V}$ and consider the following commutative diagram $(*)$ :
\begin{displaymath}
    \xymatrix{
        U \ar[r]^{{f_k}|_{U}} \ar[d]_j & U' \ar[d]^{j_k} \\
        \mathcal{V} \ar[r]_{f_k}       & \mathcal{V}_k }
\end{displaymath}
where $j:U\to \mathcal{V}$ and $j_k:U'\to \mathcal{V}_k$ are the
inclusion maps and ${f_k}|_{U}:U \to U'$ is the isomorphism induced
by ${f_k}:\mathcal{V}\to \mathcal{V}_k$. Since $\mathcal{V}_k$ is
$\mathbb{Q}$-factorial, we see that
${f_k}_*(N\mathcal{D})=\mathcal{L}'$ is a line bundle on
$\mathcal{V}_k$ for some positive integer $N$. Write
$\mathcal{L}'=\sum_h a_h {\mathcal{L}'}_h$, where ${\mathcal{L}'}_h$
are the generators of Pic$(\mathcal{V}_k)$. Then by $(*)$ we get
$$N\mathcal{D}|_{U}=j^* (N\mathcal{D})={{f_k}|_{U}}_*j^* (N\mathcal{D})={j_k}^*{f_k}_* (N\mathcal{D})=
\sum_h a_h {j_k}^*{\mathcal{L}_{h}^{'}}=$$
$$=\sum_h a_h
({f_k}|_{U})^*{j_k}^*{\mathcal{L}_{h}^{'}}=\sum_h a_h
j^*{f_k}^*{\mathcal{L}_{h}^{'}}=j^*(\sum_h a_h
{f_k}^*{\mathcal{L}_{h}^{'}})=(\sum_h a_h
{f_k}^*{\mathcal{L}_{h}^{'}})|_{U}.$$ By Hartogs' Lemma (see, e.g.,
\cite[(11.4)]{E}), this gives $N\mathcal{D}=\sum_h
a_hf_k^*(\mathcal{L}_{h}^{'})$, i.e.
$\mathcal{D}=\sum_h\frac{a_h}{N}f_k^*(\mathcal{L}_{h}^{'})$.
Therefore, if $A_k$ is not ample, then we deduce that there exists
an irreducible curve $\Gamma\subset\mathcal{V}$ such that
$m_kA_k\cdot\Gamma=0$ for any positive integer $m_k$, i.e. the map
$f_k$ contracts the curve $\Gamma$. Hence
$f_k^*(\mathcal{L}_{h}^{'})\cdot\Gamma =0$ for any $h$, i.e.
$\mathcal{D}\cdot\Gamma=0$, but this leads to a contradiction by
taking $\mathcal{D}=\mathcal{L}$.
\end{proof}

\bigskip

Similar results as in \cite[Theorem 1]{Ba} (see
also \cite[Prop.VI]{S1}) for the case of reducible ample divisors on
$\mathcal{V}$ can be now proved.

\begin{pro}\label{prop-2}
Let $\mathcal{V}$ be an integral normal complex projective variety
of dimension $n\geq 3$ and let $\mathcal{L}$ be an ample line
bundle on $\mathcal{V}$. Assume that $(\diamondsuit)$ holds. If,
up to renaming, $A_1\cong\mathbb{P}^{n-1}$, then $\mathcal{V}$ is
the cone $\mathcal{C}(\mathbb{P}^{n-1},\mathcal{O}_{\mathbb{P}^{n-1}}(s))$
on $(\mathbb{P}^{n-1},\mathcal{O}_{\mathbb{P}^{n-1}}(s))$ with
$A_1=v_s(\mathbb{P}^{n-1})$ and
$\mathcal{N}_{A_1/\mathcal{V}}\cong\mathcal{O}_{\mathbb{P}^{n-1}}(s)$
for a suitable integer $s>0$, where $v_s$ is the $s^{\textrm{th}}$
Veronese embedding of $\mathbb{P}^{n-1}$.
\end{pro}

\begin{proof}
Since $(\diamondsuit)$ holds and $A_1\cong\mathbb{P}^{n-1}$, by
Theorem \ref{thm-1} we have that $f_1(A_1)\cong\mathbb{P}^{n-1}$
is ample on $\mathcal{V}_1$. By \cite[Theorem 1]{Ba} we see that
$\mathcal{V}_1$ is the cone
$\mathcal{C}(\mathbb{P}^{n-1},\mathcal{O}_{\mathbb{P}^{n-1}}(s))$
over $(\mathbb{P}^{n-1},\mathcal{O}_{\mathbb{P}^{n-1}}(s))$, where
$s$ is a positive integer such that
$\mathcal{N}_{f_1(A_1)/\mathcal{V}_1}\cong\mathcal{O}_{\mathbb{P}^{n-1}}(s)$.

Since $\mathcal{V}_1$ is $\mathbb{Q}$-factorial and
$\mathrm{Pic}(\mathcal{V}_1)=\mathbb{Z}$, by Lemma \ref{Lemma} we
deduce that $f_1$ is an isomorphism and
$\mathrm{Pic}(\mathcal{V})=\mathbb{Z}$. Therefore, $A_1$ is an ample
divisor on $\mathcal{V}$ and by applying now \cite[Theorem 1]{Ba} to
the pair $(\mathcal{V},A_1)$, we get that
$\mathcal{V}\cong\mathcal{C}(\mathbb{P}^{n-1},\mathcal{O}_{\mathbb{P}^{n-1}}(s))$,
$A_1\cong v_s(\mathbb{P}^{n-1})$ and $\mathcal{N}_{A_1/\mathcal{V}}\cong\mathcal{O}_{\mathbb{P}^{n-1}}(s)$
for a suitable integer $s>0$.
\end{proof}

\begin{pro}\label{prop-3}
Let $\mathcal{V}$ be an integral normal Gorenstein projective
variety of dimension $n\geq 3$ and let $\mathcal{L}$ be an ample
line bundle on $\mathcal{V}$. Suppose that $(\diamondsuit)$ holds
and that $\dim\mathrm{Irr}(\mathcal{V})\leq 0$, where
$\mathrm{Irr}(\mathcal{V})$ is the set of irrational singularities
of $\mathcal{V}$. Assume that each $A_i$ is a normal Gorenstein
variety such that
$K_{A_i}+\tau_i\mathcal{H}_i\simeq\mathcal{O}_{A_i}$ for some integer $\tau_i$. If
$\mathcal{V}-A_k$ is a local complete intersection for some
$k=1,...,r$ and either $n\geq 4$, or $n=3$ and $\mathcal{V}-F$ is a
local complete intersection for some finite, possibly empty set
$F\subset\mathcal{V}-\mathrm{Irr}(\mathcal{V})$, then
$\mathrm{Pic}(\mathcal{V})=\mathbb{Z}\langle\Lambda\rangle$, where
$\Lambda$ is an ample line bundle on $\mathcal{V}$,
$K_{\mathcal{V}}=\rho\Lambda$, $A_i=a_i\Lambda$ and
$\Lambda_{A_i}=h_i\mathcal{H}_i$ with $\tau_i=-h_i(\rho+a_i)$, where $\rho$,
$a_i>0$ and $h_i>0$ are integers. In particular, for $n\geq 5$
we have $h_i=1$ for every $i=1,...,r$.
\end{pro}

\begin{proof}
Assume that
$n\geq 4$. From case $(4)$ of Corollary \ref{corollary nef and big}
it follows that either (a) $n\geq 5$ and
$\mathrm{Pic}(\mathcal{V})\cong\mathrm{Pic}(A_i)=\mathbb{Z}$ for any
$i=1,...,r$, or (b) $n=4$ and $\mathrm{Pic}(\mathcal{V})$ restricts
injectively into $\mathrm{Pic}(A_i)=\mathbb{Z}$. In both situations,
we see that
$\mathrm{Pic}(\mathcal{V})=\mathbb{Z}\langle\Lambda\rangle$ for some
ample line bundle $\Lambda$ on $\mathcal{V}$. Moreover, in (a) we
have that $\Lambda_{A_i}\simeq\mathcal{H}_i$, while in (b) we have
that $\Lambda_{A_i}\simeq h_i\mathcal{H}_i$ for some positive
integer $h_i$. By adjunction formula, we obtain that
\begin{equation}
\mathcal{O}_{A_i}\simeq K_{A_i}+\tau_i\mathcal{H}_i\simeq \left\{
\begin{array}{lr}
(K_\mathcal{V}+A_i+\tau_i\Lambda)_{A_i}\quad \ \ \text{in case (a),} \\
(K_\mathcal{V}+A_i+\frac{\tau_i}{h_i}\Lambda)_{A_i} \quad \ \
\text{in case (b)}.
\end{array}
\right. \notag
\end{equation}
i.e.
$K_\mathcal{V}+A_i+\frac{\tau_i}{h_i}\Lambda\simeq\mathcal{O}_\mathcal{V}$
for some $h_i\geq 1$. If we put $K_\mathcal{V}=\rho\Lambda$ and
$A_i=a_i\Lambda$ for some integers $\rho$ and $a_i\geq 1$, then we
see that $\tau_i=-h_i(\rho+a_i)$, with $h_i=1$ when $n\geq 5$.

As to the case $n=3$, from $(\mathit{III})$ of Corollary
\ref{corollary ample} it follows that $\mathrm{Pic}(\mathcal{V})$
restricts injectively into $\mathrm{Pic}(A_i)=\mathbb{Z}$. Then by
arguing as above, we conclude that also in this situation
$\mathrm{Pic}(\mathcal{V})=\mathbb{Z}\langle\Lambda\rangle$ for some
ample line bundle $\Lambda$ on $\mathcal{V}$ and, with the same
notation as above, $\tau_i=-h_i(\rho+a_i)$, where $h_i$ and $a_i$
are positive integers.
\end{proof}

\begin{rem}
When $n\geq 6$, Proposition \ref{prop-3} generalizes {\em \cite[Theorem 1]{T}}.
\end{rem}

\subsection{All the $A_i$'s have small degrees}\label{small degree}

Denote now by $X$ an $n$-fold with $n\geq 3$ and by $L$ an ample line bundle on $X$.
In this subsection we prove the last result stated in the Introduction.

\smallskip

\noindent\textit{Proof of Proposition \ref{prop-1}}. Since
$[H_i]_{A_i}^{n-1}\leq 4$, by \cite[(8.10.1)]{BS} we see that
$(A_i,[H_i]_{A_i})$ satisfies one of the following two conditions:
\begin{enumerate}
\item[(a)] Pic$(A_i)=\mathbb{Z}\langle [H_i]_{A_i}\rangle$;
\item[(b)] $(A_i,[H_i]_{A_i})\cong
(\mathbb{P}^1\times\mathbb{P}^{3},\an_{\mathbb{P}^1\times\mathbb{P}^{3}}(1,1))$.
\end{enumerate}

\medskip

{\bf Step (I)}. First of all, assume that $r=1$. In Case (a), by Lefschetz's Theorem
we see that Pic$(X)=\mathbb{Z}\langle H_1\rangle$. Write
$A_1=a_1H_1$ for a positive integer $a_1$. Since
$a_1H_1^n=A_1H_1^{n-1}=[H_1]_{A_1}^{n-1}\leq 4$, we deduce that
$a_1H_1^n\leq 4$. Consider the map $\varphi :X\to\mathbb{P}^N$
associated to $|H_1|$. Since $|H_1|$ is ample and spanned, the
morphism $\varphi$ is finite and such that
$$[H_1]^n=\deg\varphi\cdot\deg\varphi(X)\leq 4.$$
This gives the following three possibilities:
\begin{enumerate}
\item[$(1)$] $\deg\varphi=1$,
$\deg\varphi(X)\leq 4$; \ $(2)$ $\deg\varphi=2$,
$\deg\varphi(X)\leq 2$; \
$(3)$ $\deg\varphi=3,4$ and
$\deg\varphi(X)=1$.
\end{enumerate}
In $(1)$, since $n\geq 5$ and $\mathrm{Pic}(X)=\mathbb{Z}\langle
H_1\rangle$, by \cite[(8.10.1)]{BS} (see also \cite{F1}, \cite{I2},
\cite{I}) we deduce that $(X,H_1)$ is one of the following pairs:
\begin{itemize}
\item $(\mathbb{P}^n,\mathcal{O}_{\mathbb{P}^n}(1))$, where $A_1\in
|\mathcal{O}_{\mathbb{P}^n}(a_1)|$ with $0<a_1\leq 4$; \item
$(\mathbb{Q}^n,\mathcal{O}_{\mathbb{Q}^n}(1))$, where $A_1\in
|\mathcal{O}_{\mathbb{Q}^n}(a_1)|$ with $0<a_1\leq 2$; \item $(V_d,
\mathcal{O}_{\mathbb{P}^{n+1}}(1)_{V_d})$ with $A_1\in |H_1|$, where
$V_d\subset\mathbb{P}^{n+1}$ is a smooth hypersurface of degree
$d=3,4$; \item $(W, \mathcal{O}_{\mathbb{P}^{n+2}}(1)_W)$ with
$A_1\in |H_1|$, where
$W=\mathbb{Q}_1\cap\mathbb{Q}_2\subset\mathbb{P}^{n+2}$ is a
complete intersection of two quadric hypersurfaces
$\mathbb{Q}_i\subset\mathbb{P}^{n+2}$ for $i=1,2$.
\end{itemize}
In $(2)$, the map $\varphi$ is a double cover of either (i)
$\mathbb{P}^n$, or (ii) $\mathbb{Q}^n\subset\mathbb{P}^{n+1}$. In
case (i), we see that $H_1=\varphi^*\mathcal{O}_{\mathbb{P}^n}(1)$
and $A_1\in |a_1H_1|$ with $a_1=1,2$. In case (ii), we get
$H_1=\varphi^*\mathcal{O}_{\mathbb{Q}^n}(1)$ and $A_1\in |H_1|$.
Finally, in $(3)$ the morphism $\varphi$ is a $d$-cover of
$\mathbb{P}^n$ with $d=3,4$, and
$A_1=H_1=\varphi^*\mathcal{O}_{\mathbb{P}^n}(1)$.

\smallskip

Finally, consider Case (b). Note that
$\mathrm{Pic}(X)\neq\mathbb{Z}$. Moreover, since $A_1$ is ample,
for any ample line bundle $H$ on $\mathbb{P}^1\times\mathbb{P}^3$, we have
$$0=\mathcal{O}_{\mathbb{P}^1\times\mathbb{P}^3}(1,0)\cdot\mathcal{O}_{\mathbb{P}^1\times\mathbb{P}^3}(2,0)\cdot H^2=
\mathcal{O}_{\mathbb{P}^1\times\mathbb{P}^3}(1,0)\cdot(K_{A_1}+4{H_1}_{A_1})\cdot H^2=$$
$$=\mathcal{O}_{\mathbb{P}^1\times\mathbb{P}^3}(1,0)\cdot
(K_X+4H_1+A_1)_{A_1}\cdot H^2\geq
\mathcal{O}_{\mathbb{P}^1\times\mathbb{P}^3}(1,0)\cdot
(K_X+5H_1)_{A_1}\cdot H^2,$$
i.e. $K_X+5H_1$ can not be ample on $X$.
So by \cite{I1}, \cite{F1} and
\cite[$\S 7.2$]{BS},
we see that $(X,H_1)$ is a scroll
over $\p^1$ of dimension five, i.e. $X\cong\mathbb{P}(\mathcal{E})$,
where $\mathcal{E}\cong \mathcal{O}_{\mathbb{P}^1}(a_1)\oplus
...\mathcal{O}_{\mathbb{P}^1}(a_5)$ with $1\leq a_1\leq ...\leq
a_5,$ and $H_1$ is the tautological line bundle $\xi$ of
$\mathbb{P}(\mathcal{E})$. Put $\xi'=\xi-a_1F$. Since $A_1$ is ample
on $X$, write $A_1=a\xi'+bF$ for some positive integers $a$ and $b$
(see \cite[(3.2.4)]{BS}). Then
$$[H_1]_{A_1}^{4}=(a\xi'+bF)(\xi)^4=a(\xi-a_1F)(\xi)^4+b=a(a_2+...+a_5)+b\geq
5,$$ but this is absurd. Thus Case (b) can not occur for $r=1$.

\medskip

{\bf Step (II)}. Suppose now that $r\geq 2$.
First of all, let us prove the following

\bigskip

\noindent\textit{Claim.} If $(A_{1},{H_1}_{A_{1}})\cong
(\mathbb{P}^1\times\mathbb{P}^{3},\an_{\mathbb{P}^1\times\mathbb{P}^{3}}(1,1))$,
then $K_{X}+4H_i$ is not nef for some $i=1,...,r$.

\smallskip

\noindent Suppose that $K_{X}+4H_k$ is nef for every $k=1,...,r$.
Let $\Phi _{1}:A_{1}\to\mathbb{P}^1$ be the nefvalue morphism of
$(A_{1},{H_1}_{A_{1}})$. Up to renaming, assume that
$A_{2},...,A_{s}$ are the only components of $A$ such that
$A_{1}\cap A_{h}\neq\varnothing $ for $h=2,...,s$ with $s\leq r$.
Put $h_{k}:=[A_{k}]_{A_{1}}$ for $k=2,...,r$ and note that $h_{t}$
is trivial for $t=s+1,...,r$.

First of all, if $\Phi _{1}(h_{k})$ is a union of points of
$\mathbb{P}^1$ for every $k=2,...,s$, then for a general fiber
$F\cong\mathbb{P}^3$ of $\Phi _{1}$ we have that
\begin{equation*}
(L_{A_{1}})_{F}=([A_{1}]_{A_{1}})_{F}+[h_{2}]_{F}+...+[h_{s}]_{F}=([A_{1}]_{A_{1}})_{F}=%
\mathcal{O}_{F}(a),
\end{equation*}
for a suitable integer $a\geq 1$. Moreover, note that
$K_{A_{1}}+4{H_1}_{A_{1}}\simeq 2F$.

Thus by adjunction we obtain that
$$[K_{X}+4H_1]_{F}{H_1}_{F}^{2}=([K_{X}+4H_1]_{A_{1}})_{F}({H_1}_{A_{1}})_{F}^{2}=$$
$$=(2F-[A_{1}]_{A_{1}})_{F}({H_1}_{A_{1}})_{F}^{2}=-([A_{1}]_{A_{1}})_{F}({H_1}_{A_{1}})_{F}^{2}
=-\mathcal{O}_{F}(a)\mathcal{O}_{F}(1)^{2}<0,$$ but this is absurd,
since $K_{X}+4H_1$ is nef and $H_1$ is ample on $X$.

Assume now that $\Phi _{1}(h_{t})=\mathbb{P}^1$ for some
$t=2,...,s$. If $\mathrm{Pic}(A_t)=\mathbb{Z}$, then by
\cite[(4.2)]{CHS} we get a contradiction by taking $B=A_t$ and
$A=A_1$. Thus we can assume that $(A_{t},{H_t}_{A_{t}})\cong
(\mathbb{P}^1\times\mathbb{P}^{3},\an_{\mathbb{P}^1\times\mathbb{P}^{3}}(1,1))$.
Note that there exists a general fiber $F$ of $\Phi _{1}$ such that
$F\nsubseteq h_{t}$ and $[h_{t}]_{F}$ is effective and not trivial.
Moreover, we have that Pic$(F)\cong\mathbb{Z}$ and $F\nsubseteq
A_{t}$, since otherwise $F\subseteq A_{1}\cap A_{t}=h_{t}$.
Therefore $F\cap A_{t}\neq \varnothing $ and $F\cap A_{t}\neq F$.
Since
\begin{equation*}
\dim F+\dim A_{t}-\dim\Phi_1(A_1)=3+4-1=6>5=\dim X,
\end{equation*}
by the same argument as in the proof of \cite[(4.2)]{CHS} with $B=F$
and $A=A_t$, we can conclude that $F\subseteq A_{t}$, but this gives
a contradiction. \hfill{Q.E.D.}

\bigskip

\noindent By the above Claim, if one of the $A_k$, say $A_1$, is
such that $$(A_{1},{H_1}_{A_{1}})\cong
(\mathbb{P}^1\times\mathbb{P}^{3},\an_{\mathbb{P}^1\times\mathbb{P}^{3}}(1,1)),$$
then $n=5$ and $K_{X}+4H_i$ is not nef for some $i=1,...,r$. Since
$\mathrm{Pic}(A_1)\neq\mathbb{Z}$,
by \cite{I1}, \cite{F1} and \cite[\S 7]{BS}
we deduce that $(X,H_i)$ is a $\mathbb{P}^4$-bundle over a smooth
curve $C$. Moreover, $A_1$ dominates $C$. So $C\cong\mathbb{P}^1$
and $X\cong\mathbb{P}(\mathcal{E})$ for some vector bundle
$\mathcal{E}$ of rank-5 on $\mathbb{P}^1$ such that
$\mathcal{E}\cong\mathcal{O}_{\mathbb{P}^1}\oplus\mathcal{O}_{\mathbb{P}^1}(a_1)\oplus
...\oplus\mathcal{O}_{\mathbb{P}^1}(a_4)$ with $0\leq a_1\leq
...\leq a_4$. Let $\xi$ be the tautological line bundle of
$\mathbb{P}(\mathcal{E})$. Write $A_1=a_1\xi+b_1F$ and
$H_1=\alpha_1\xi+\beta_1F$ with $b_1\geq 0$ and
$a_1,\alpha_1,\beta_1$ positive integers (see \cite[(3.2.4)]{BS}).
Furthermore, put $F_{A_1}\simeq\mathcal{O}_{A_1}(a,0)$ and
$\xi_{A_1}\simeq\mathcal{O}_{A_1}(b,c)$ with $a>0$ and $b,c\geq 0$.
By the adjunction formula we have that
$$\mathcal{O}_{A_1}(-2,-4)\simeq K_{A_1}\simeq
(K_X+A_1)_{A_1}\simeq [-5\xi
+(\deg\mathcal{E}-2)F+a_1\xi+b_1F]_{A_1}=$$
$$=(a_1-5)\xi_{A_1}+(b_1+\deg\mathcal{E}-2)F_{A_1}\simeq\mathcal{O}_{A_1}(b(a_1-5)+a(b_1+\deg\mathcal{E}-2),c(a_1-5)).$$
This gives the following two equations:
\begin{equation}
4=c(5-a_1) \tag{$\flat$}
\end{equation}
\begin{equation}
2=b(5-a_1)-a(b_1+\deg\mathcal{E}-2). \tag{$\natural$}
\end{equation}
Note that from $(\flat$) it follows that $1\leq a_1\leq 4$ and then
$$4\geq a_1=[A_1]_F\cdot [\xi]_F^3=F_{A_1}\cdot
{\xi}_{A_1}^3=\mathcal{O}_{A_1}(a,0)\cdot\mathcal{O}_{A_1}(b,c)^3=ac^3.$$
Thus we deduce that $c=1$ and $a_1=a$. By ($\flat$) we have also
that $a=a_1=1$ and the equation ($\natural$) gives
$4b=b_1+\deg\mathcal{E}$. Since
$$\mathcal{O}_{A_1}(1,1)\simeq
[H_1]_{A_1}=[\alpha_1\xi+\beta_1F]_{A_1}=\mathcal{O}_{A_1}(\alpha_1b+\beta_1,\alpha_1),$$
we see that $\alpha_1=1$ and $1=b+\beta_1\geq b+1\geq 1$, i.e. $b=0$
and $\beta_1=1$. Therefore, by $(\natural)$ we get
$0=4b=b_1+\deg\mathcal{E}\geq\deg\mathcal{E}$, i.e. $a_1=...=a_4=0$
and $b_1=0$. This shows that
$X\cong\mathbb{P}(\mathcal{O}_{\mathbb{P}^1}^{\oplus
5})\cong\mathbb{P}^1\times\mathbb{P}^4$, $A_1\in
|\mathcal{O}_X(0,1)|$ and $H_1\in |\mathcal{O}_X(1,1)|$.
Consider $A_i\in |\mathcal{O}_X(d,c)|$ and $H_i\in |\mathcal{O}_X(t,s)|$
for some $i=2,...,r$, where $d,c\geq 0$ and $t,s>0$. Then we have
$$4\geq [H_i]_{A_i}^4=\mathcal{O}_X(d,c)\cdot\mathcal{O}_X(t,s)^4=s^3(ds+4ct)\geq s^3.$$
This gives $s=1$ and $0\leq c\leq 1$. If $c=1$, then we get $d=0, t=s=c=1$ and $(A_i,[H_i]_{A_i})$
is of the same type as $(A_1,[H_1]_{A_1})$. If $c=0$, then we see that $1\leq d\leq 4$, $t\geq 1$ and $s=1$.
This shows that $A_i\in |\mathcal{O}_X(d,0)|$ and $H_i\in |\mathcal{O}_X(t,1)|$ with $t\geq 1$ and
$1\leq d\leq 4$, i.e. $A_i\to\mathbb{P}^4$ is a $d$-cover of $\mathbb{P}^4$ with $1\leq d\leq 4$.
Note that in this case we have
$$|\mathcal{O}_X(d,0)|=\underbrace{|\mathcal{O}_X(1,0)|+...+|\mathcal{O}_X(1,0)|}_{d-\mathrm{times}}.$$
Thus, since $A_i\in |\mathcal{O}_X(d,0)|$ is irreducible and reduced, we conclude that $d=1$.

\smallskip

Finally, since the above argument works independently from the choice of the
component $A_1$ of $A\in |L|$, we can assume, without loss of
generality, that every component $A_i$ of $A$ is such that
$\mathrm{Pic}(A_i)=\mathrm{Z}\langle [H_i]_{A_i}\rangle$. Then by
Proposition \ref{prop-3} and Corollary \ref{corollary ample}
(\textit{III}), we conclude that
$\mathrm{Pic}(X)=\mathbb{Z}\langle\Lambda\rangle$ for some ample
line bundle $\Lambda$ on $X$. Write $A_i=a_i\Lambda$ and
$H_i=b_i\Lambda$ for suitable positive integers $a_i$ and $b_i$.
Thus we obtain that
$$a_ib_i^{n-1}\Lambda^n=
(a_i\Lambda)(b_i\Lambda)^{n-1}=[H_i]_{A_i}^{n-1}\leq 4,$$ i.e.
$H_1=...=H_r=\Lambda$ and $a_i\Lambda^n\leq 4$. Consider the map
$\varphi:X\to\mathbb{P}^N$ associated to $|\Lambda|$. Since
$\Lambda$ is ample and spanned, we have
$\Lambda^n=\deg\varphi\cdot\deg\varphi (X)\leq 4$ and by arguing
as in case $r=1$, we obtain the statement. \hfill{$\square$}

\bigskip
\bigskip

\noindent\textbf{Acknowledgments.} The author would thank the
referee(s) for kind comments and useful suggestions which improved
the final form of this paper.

\bigskip
\bigskip

{\small

\bibliographystyle{amsplain}

}

\bigskip
\bigskip
\medskip

\noindent Departamento de Matematica, Universidad de Concepci\'on,
Casilla 160-C, Concepci\'on, Chile.

\medskip

\noindent E-mail:\ atironi@udec.cl

\end{document}